\documentclass[11pt]{amsart}
\usepackage{amsfonts,amssymb,amscd,amstext,mathrsfs,mathtools}
\usepackage[utf8]{inputenc}
\usepackage{hyperref}
\usepackage{verbatim}

\usepackage{graphics}
\usepackage{graphicx}

\usepackage{times}
\usepackage{enumerate}
\usepackage[up,bf]{caption}
\usepackage{color}
\input xy
\xyoption{all}

\numberwithin{equation}{section}

\newtheorem{theorem}{Theorem}[section] 
\newtheorem{proposition}[theorem]{Proposition}

\newtheorem{corollary}[theorem]{Corollary} 

\theoremstyle{definition} 
\newtheorem{definition}[theorem]{Definition} 
\newtheorem{remark}[theorem]{Remark}

\newcommand\Cscr{\mathscr{C}}

\newcommand\Oscr{\mathscr{O}}

\newcommand\C{\mathbb{C}} 
\newcommand\D{\overline{\mathbb D}} 
 
\renewcommand\D{\mathbb D}

\newcommand\N{\mathbb{N}} 
 
\newcommand\R{\mathbb{R}}

\newcommand\U{\mathbb{U}} 
\newcommand\Z{\mathbb{Z}}

\newcommand\igot{\mathfrak{i}}

\renewcommand\igot{\mathfrak{i}}

\newcommand\E{\mathrm{e}} 
\newcommand\I{\imath} 
\renewcommand\imath{\igot}

\newcommand\lra{\longrightarrow} 
 
\newcommand\longhookrightarrow{\ensuremath{\lhook\joinrel\relbar\joinrel\rightarrow}} 

\newcommand\di{\partial} 
\newcommand\dibar{\overline\partial}

\newcommand\Id{\mathrm{Id}}

\def\Ell1{\mathrm{Ell_1}} 
\def\CEll1{\mathrm{CEll_1}}

\def\Jst{J_{\mathrm{st}}}

\newcommand\Pic{\mathrm{Pic}} 
\newcommand\Hom{\mathrm{Hom}}

\makeatletter
\@namedef{subjclassname@2020}{\textup{2020} Mathematics Subject Classification}
\makeatother

\begin{document} 

\title{The nonhomogeneous Cauchy--Riemann equation on \\ families of open Riemann surfaces}

\author{Franc Forstneri{\v c}}

\address{Franc Forstneri\v c, Faculty of Mathematics and Physics, University of Ljubljana, and Institute of Mathematics, Physics, and Mechanics, Jadranska 19, 1000 Ljubljana, Slovenia}
\email{franc.forstneric@fmf.uni-lj.si}



\date{19 August 2025}


\begin{abstract}
In this paper we use the nonhomogeneous Beltrami equation 
to give an optimal solution to the nonhomogeneous Cauchy--Riemann equation
for continuous or smooth families of complex structures and $(0,1)$-forms
of a H\"older class on a smooth open orientable surface. As an application, we obtain the Oka--Grauert principle for complex line bundles on families
of open Riemann surfaces. 
\end{abstract}

\maketitle

%
%
\section{The $\dibar$-equation on families of open Riemann surfaces}
\label{sec:dibar}

In this paper, we show that the nonhomogeneous Cauchy--Riemann
equation, or the $\dibar$-equation for short, can be solved for very
general families of complex structures and $(0,1)$-forms of a H\"older class
on a smooth open orientable surface, with the usual gain of one derivative 
in the space variable and without any loss of regularity in the parameter. 
Our main results are Theorem \ref{th:dibar} and its global version, 
Corollary \ref{cor:dibar}. The proof uses the
nonhomogeneous Beltrami equation 
together with the Runge approximation theorem for families
of holomorphic functions on families of open Riemann surfaces
(see \cite[Theorem 1.1]{Forstneric2024Runge}). An application 
is a Dolbeault cohomology vanishing theorem 
(see Proposition \ref{prop:Dolbeault}) and the Oka principle
for complex line bundles on such families (see Theorem \ref{th:OPLB}).

We begin by introducing the setup, referring
to \cite{Forstneric2024Runge} for more details.
Let $X$ be a smooth orientable surface.
A complex structure on $X$ is an endomorphism $J$ of its tangent bundle 
$TX$ satisfying $J^2:=J\circ J = -\Id$. Thus, $J$ is a section of the smooth 
vector bundle $T^*X\otimes TX\cong \Hom(TX,TX) \to X$ whose fibre over 
$x\in X$ is the space $\Hom(T_xX,T_xX)$ of linear maps $T_x X\to T_xX$.
We endow the plane $\R^2\cong \C$ with the standard complex
structure $\Jst$ given in standard basis by the matrix
$\left(\begin{smallmatrix} 0 & -1\\ 1 & 0\end{smallmatrix}\right)$, 
and in complex notation by multiplication by $\imath=\sqrt{-1}$.  
A differentiable function $f:U\to\C$ on an open set $U\subset X$ 
is said to be {\em $J$-holomorphic} (or $(J,\Jst)$-holomorphic) 
if the Cauchy--Riemann equation 
$df_x \circ J_x = \Jst \circ df_x=\imath\, df_x$ holds 
for every $x\in U$. We say that $J$ is of local H\"older class 
$\Cscr^{(k,\alpha)}$ for some $k\in\Z_+=\{0,1,2,\ldots\}$ and $0<\alpha<1$
if for any relatively compact domain $\Omega\Subset X$, the restriction
$J|_\Omega\in \Gamma^{(k,\alpha)}(\Omega,T^*\Omega\otimes T\Omega)$ 
is a section of $T^*\Omega\otimes T\Omega$ of class 
$\Cscr^{(k,\alpha)}(\Omega)$, i.e., with finite $\Cscr^{(k,\alpha)}(\Omega)$
norm. (H\"older norms are defined 
with respect to a smooth Riemannian metric on $X$;
see \cite[Sect. 4.1]{GilbargTrudinger1983}.
If $\Omega$ has $\Cscr^1$ boundary, as will be the case 
in our results, then a function $f\in \Cscr^{(k,\alpha)}(\Omega)$
extends to a unique function in $\Cscr^{(k,\alpha)}(\overline\Omega)$.
The same holds for complex structures.) 
For such $J$, there is an atlas $\{(U_i,\phi_i)\}_i$ of 
open sets $U_i\subset X$ with $\bigcup_{i} U_i=X$ and 
$J$-holomorphic charts $\phi_i :U_i \to \phi_i(U_i)\subset \C$ 
of class $\Cscr^{(k+1,\alpha)}(U_i)$; 
see \cite[Theorem 5.3.4]{AstalaIwaniecMartin2009}. 
Since the transition maps $\phi_i\circ\phi_j^{-1}$
are $\Jst$-biholomorphic, 
$J$ determines on $X$ the structure of a Riemann surface,  
and every $J$-holomorphic function is of local class $\Cscr^{(k+1,\alpha)}$
in the given smooth structure on $X$. Since the inverse
of a diffeomorphism of local class $\Cscr^{(k+1,\alpha)}$
is of the same class (see Norton \cite{Norton1986} and 
Bojarski et al.\ \cite[Theorem 2.1]{BojarskiAll2005}), 
the smooth structure on $X$ determined by a complex structure $J$ 
of local class $\Cscr^{(k,\alpha)}$ is $\Cscr^{(k+1,\alpha)}$ compatible 
with the given smooth structure. In the sequel, we shall often admit
the adjective "local" when speaking of complex structures or functions of
H\"older classes. 

Let $B$ be a topological space whose nature will depend 
on the integer $l\ge 0$ to be specified.
When $l=0$, we assume that $B$ is a paracompact Hausdorff space, 
and if $l>0$ then $B$ will be a (paracompact) manifold of class $\Cscr^l$. 
A family $\{J_b\}_{b\in B}$ of complex structures on $X$
is said to be of class $\Cscr^{l,(k,\alpha)}$ if 
for any relatively compact domain $\Omega\Subset X$, the map 
$B\ni b\mapsto J_b|_{\Omega}\in \Gamma^{(k,\alpha)}
(\Omega,T^*\Omega\otimes T\Omega)$ is of class $\Cscr^l$
as a map to the H\"older space 
$\Gamma^{(k,\alpha)}(\Omega,T^*\Omega\otimes T\Omega)$.
Following Kodaira and Spencer \cite{KodairaSpencer1958}
and Kirillov \cite{Kirillov1976}, the collection $\{(X,J_b)\}_{b\in B}$
is called a family of Riemann surfaces of class 
$\Cscr^{l,(k,\alpha)}$. Such a family $\{J_b\}_{b\in B}$ can equivalently 
be given by a family $\{\mu_b\}_{b\in B}$ of maps from 
$X$ to the unit disc $\D=\{\zeta\in\C:|\zeta|<1\}$
of the same smoothness class $\Cscr^{l,(k,\alpha)}$;
see \cite[Sect.\ 2]{Forstneric2024Runge}. 
A continuous map $f:B\times X\to Y$ to a complex manifold $Y$ 
is said to be {\em fibrewise holomorphic} or {\em $X$-holomorphic} 
if the map $f_b=f(b,\cdotp):X\to Y$ is $J_b$-holomorphic
for every $b\in B$. Assuming that $\{J_b\}_{b\in B}$ is of class
$\Cscr^{l,(k,\alpha)}$, the space $Z=B\times X$ 
endowed with the complex structure
$J_b$ on the fibre $\{b\}\times X$ admits fibre preserving 
$X$-holomorphic charts of class $\Cscr^{l,(k+1,\alpha)}$ 
with values in $B\times \C$ which are local in $b\in B$
and semiglobal in the space variable $x\in X$
(see \cite[Theorem 4.1]{Forstneric2024Runge}), and 
every $X$-holomorphic function $f:B\times X\to\C$ of class 
$\Cscr^{l,0}$ (i.e., $f$ admits continuous partial derivatives 
of order $\le l$ with respect to the parameter $b\in B$) 
is of local class $\Cscr^{l,(k+1,\alpha)}$
(see \cite[Lemma 5.6]{Forstneric2024Runge}). 

Assume now that $X$ is a smooth {\em open} surface endowed
with a complex structure $J$ of class $\Cscr^{(k,\alpha)}$
$(k\in \Z_+,\ 0<\alpha<1)$.
By Gunning and Narasimhan \cite{GunningNarasimhan1967} 
there is a $J$-holomorphic immersion $z:X\to \C$.
By what was said above, $z$ is of local class 
$\Cscr^{(k+1,\alpha)}$ in the given smooth structure on $X$, and 
it provides a local holomorphic coordinate on $X$ at every point. 
Its differentials $dz$ and $d\bar z$ are of class $\Cscr^{(k,\alpha)}$, 
and they trivialise the respective cotangent bundles 
$T^{*(1,0)}_J X$ and $T^{*(0,1)}_J X$ 
consisting of $(1,0)$-forms and $(0,1)$-forms with respect to $J$. 
We have $\C\otimes_\R T^*X= T^{*(1,0)}_J X\oplus T^{*(0,1)}_J X$,
and every $1$-form $\beta$ on $X$ can be uniquely written
as $\beta = A dz + B d\bar z$ for a pair of functions $A,B:X\to \C$.
Note that $\beta$ is of class $\Cscr^{(k,\alpha)}$ if and only if 
the functions $A,B$ are of class $\Cscr^{(k,\alpha)}$. 
Given a differentiable function $f:X\to \C$, its differential splits
in the $(1,0)$ and $(0,1)$ parts: 
\begin{equation}\label{eq:derivatives}
	df=\di_J f+\dibar_J f=f_z \,\cdotp dz + f_{\bar z} \,\cdotp d\bar z.  
\end{equation}
Given a $(0,1)$-form $\beta=u \, d\bar z$ on a domain $\Omega\subset X$,
the $\dibar_J$-equation asks for a function $f:\Omega\to \C$ 
satisfying $\dibar_J f=\beta$; equivalently, $f_{\bar z}=u$.
If $\Omega$ is relatively compact and has sufficiently regular boundary
then this elliptic equation is solvable in many function
spaces with a gain of one derivative; see e.g.\  
Ahlfors \cite{Ahlfors2006} and Astala et al.\ \cite{AstalaIwaniecMartin2009}.

We shall prove the following result, which gives families 
of solutions of the $\dibar$-equation 
for families of complex structures and $(0,1)$-forms
on domains in a smooth open surface. 

%
%
\begin{theorem}\label{th:dibar}
Assume that $X$ is a smooth open orientable surface, 
$\Omega\Subset X$ is a relatively compact domain with 
$\Cscr^{(k+1,\alpha)}$ boundary for some $k\in\Z_+$ and
$0<\alpha<1$, $l\in \Z_+$, $B$ is a paracompact Hausdorff
space if $l=0$ and a $\Cscr^l$ manifold if $l>0$,
and $\{J_b\}_{b\in B}$ is a family of complex structures 
of class $\Cscr^{l,(k,\alpha)}(B\times \overline\Omega)$ 
on $\overline\Omega$.
Given a family of $(0,1)$-forms $\{\beta_b\}_{b\in B}$ 
on $\overline\Omega$ of class 
$\Cscr^{l,(k,\alpha)}(B\times \overline\Omega)$, 
there is a function 
$f\in \Cscr^{l,(k+1,\alpha)}(B\times\overline\Omega)$ satisfying
\begin{equation}\label{eq:dibar}
	\dibar_{J_b} f(b,\cdotp) = \beta_b \ \ 
	\text{on $\overline \Omega$ for every $b\in B$.}
\end{equation}
\end{theorem}

Theorem \ref{th:dibar}, together with the Runge approximation theorem in
\cite[Theorem 1.1]{Forstneric2024Runge}, implies the following corollary
concerning solutions of families of global $\dibar$-equations.

%
%
\begin{corollary}\label{cor:dibar}
Assume that $\{J_b\}_{b\in B}$ is a family of complex structures 
of class $\Cscr^{l,(k,\alpha)}$ on a smooth open orientable surface $X$,
where $l,k\in \Z_+$, $l\le k+1$, $0<\alpha<1$,  
and $B$ is as in Theorem \ref{th:dibar}. 
Given a family $\{\beta_b\}_{b\in B}$ of $(0,1)$-forms 
$\beta_b\in \Gamma(X,T^{*(0,1)}_{J_b}X)$ of class 
$\Cscr^{l,(k,\alpha)}$, there is a function $f:B\times X\to\C$
of class $\Cscr^{l,(k+1,\alpha)}$ satisfying
\begin{equation}\label{eq:dibar1}
	\dibar_{J_b} f(b,\cdotp) = \beta_b \ \ 
	\text{on $X$ for every $b\in B$.}
\end{equation}
\end{corollary}

The condition $l\le k+1$ in the corollary is due to the use of the Runge
approximation theorem for fibrewise holomorphic functions
on families of open Riemann surfaces 
\cite[Theorem 1.1]{Forstneric2024Runge}, which 
was proved under this assumption. 
See also Remark \ref{rem:Cartan} concerning the $\Cscr^\infty$ case. 
This restriction is unnecessary in Theorem \ref{th:dibar}.

Except perhaps for the condition $l\le k+1$ in the corollary,
these results are optimal since we have the expected gain of one derivative 
in the space variable and no loss of regularity
in the parameter $b\in B$. Note that the 1-forms 
$\beta_b$ and the function $f$ in \eqref{eq:dibar}, \eqref{eq:dibar1}
are expressed in terms of the smooth coordinates on $X$. Expressing
them in terms of local fibrewise holomorphic coordinates leads 
to a loss of $2l$ derivatives in the space variable if $l>0$; 
see \cite[V3, Theorem 9.1]{Forstneric2024Runge}.

\begin{proof}[Proof of Theorem \ref{th:dibar}]
We shall use the connection between complex
structures and Beltrami multipliers; see 
\cite[Sect.\ 2]{Forstneric2024Runge} or any 
standard text on quasiconformal maps for the details.

Choose a smooth complex structure $J$ on $X$ and 
a $J$-holomorphic immersion $z:X\to\C$.
Let $\Omega\Subset X$ be as in the theorem. 
Every function $\mu:\overline \Omega\to\D$ 
of class $\Cscr^{(k,\alpha)}(\overline \Omega)$  with values
in the unit disc $\D$ 
determines a complex structure $J_\mu$ on $\overline\Omega$ 
of the same class, with $\mu=0$ corresponding to $J|_{\overline\Omega}$. 
Conversely, every complex structure $J'$ on $\overline \Omega$
of class $\Cscr^{(k,\alpha)}$ in the same orientation class 
as $J$ equals $J_\mu$ for a unique 
$\mu\in \Cscr^{(k,\alpha)}(\overline \Omega,\D)$. 
A differentiable function $f:\overline\Omega\to\C$ satisfies 
$\dibar_{J_\mu}f=0$, and hence is $J_\mu$-holomorphic on
$\Omega$, if and only if it satisfies the homogeneous Beltrami equation 
$f_{\bar z} = \mu f_z$, where the partial derivatives $f_z$ and 
$f_{\bar z}$ are defined by \eqref{eq:derivatives}. The analogous
statements hold for globally defined complex structures on $X$
and function $\mu:X\to \D$.

To prove the theorem, it suffices to show that for every point $b_0\in B$ 
there is a neighbourhood $B_0\subset B$ of $b_0$ such that 
equation \eqref{eq:dibar} is solvable on $B_0\times \overline\Omega$ with 
$f\in \Cscr^{l,(k+1,\alpha)}(B_0\times \overline\Omega)$. 
By using $\Cscr^l$ partitions of unity on $B$ we then obtain 
a solution on $B\times \overline \Omega$, thereby proving the theorem.

Note that the complex structure $J_{b_0}$ on $\overline\Omega$
extends to a complex structure on $X$ of local class $\Cscr^{(k,\alpha)}$.
To see this, we represent 
$J_{b_0}$ in terms of $J$ by a Beltrami multiplier 
$\mu\in \Cscr^{(k,\alpha)}(\overline \Omega,\D)$. 
Since $b\Omega$ is of class $\Cscr^{(k+1,\alpha)}$,
$\mu$ extends from $\overline \Omega$ to a function
$\mu':X\to \D$ of class $\Cscr^{(k,\alpha)}$ 
with compact support 
(see Gilbarg and Trudinger \cite[Lemma 6.37]{GilbargTrudinger1983}). 
The associated complex structure $J_{\mu'}$ on $X$ is of local 
class $\Cscr^{(k,\alpha)}$ and it coincides with $J_\mu$ on 
$\overline \Omega$. For $k\ge 1$ the same conclusion 
holds if $b\Omega$ is of class $\Cscr^{(k,\alpha)}$.

This reduces the proof to the following proposition.
In this result, the smooth structure on $X$ is induced  
by the complex structure $J$, and the H\"older norms are 
computed with respect to a fixed smooth Riemannian metric on $X$
whose precise choice is unimportant. 

%
%
\begin{proposition}\label{prop:dibar}
Let $(X,J)$ be an open Riemann surface, 
and let $\Omega\Subset X$, $k$, and $\alpha$ be as in Theorem \ref{th:dibar}.  
For $\mu\in \Cscr^{(k,\alpha)}(\overline \Omega,\D)$ let $J_\mu$ 
denote the associated complex structure on $\overline\Omega$, 
with $J_0=J|_{\overline \Omega}$. For $c>0$ set 
$
	B_c=\{\mu \in \Cscr^{(k,\alpha)}(\overline\Omega) : 
	\|\mu\|_{k,\alpha} <c\}.
$
There exists $c>0$ such that for any map 
$B_c\ni \mu \mapsto \beta_\mu \in  
\Gamma^{(k,\alpha)}(\overline\Omega,T^{*(0,1)}_{J_\mu}\overline\Omega)$
of class $\Cscr^l$, $l\in \{0,1,\ldots,\infty,\omega\}$, 
there is a function 
$f\in \Cscr^{l,(k+1,\alpha)}(B_c\times \overline\Omega)$ 
such that for every $\mu \in B_c$
the function $f_\mu=f(\mu,\cdotp):\overline\Omega\to\C$ satisfies 
\begin{equation}\label{eq:dibar2}
	\dibar_{J_\mu} f_\mu = \beta_\mu. 
\end{equation}
\end{proposition}

\begin{proof}
We begin by recalling some
technical tools from \cite[Secs.\ 3-4]{Forstneric2024Runge}.
Choose a $J$-holomorphic immersion $z:X\to\C$.  
There is a Cauchy kernel on $(X,J)$ which determines on any 
smoothly bounded domain $\Omega\Subset X$ a pair 
of bounded linear operators 
$P:\Cscr^{(k,\alpha)}(\overline \Omega) \to \Cscr^{(k+1,\alpha)}(\overline \Omega)$ and 
$S:\Cscr^{(k,\alpha)}(\overline \Omega) \to \Cscr^{(k,\alpha)}(\overline \Omega)$
($k\in\Z_+$, $0<\alpha<1$) such that 
for any $\phi\in \Cscr^{(k,\alpha)}(\overline \Omega)$, the 
Cauchy operator $P$ solves the equation
$P(\phi)_{\bar z}=\phi$ on $\overline \Omega$ and 
the Beurling operator $S$ is given by $S(\phi)=P(\phi)_z$. 
The properties of these operators are summarised in 
\cite[Theorem 3.2]{Forstneric2024Runge}. 
Although the cited result is stated for domains $\Omega$
with $\Cscr^\infty$ boundaries, it is clear from the proof and 
\cite[Lemma 6.37]{GilbargTrudinger1983} that it holds 
if $b\Omega$ is of class $\Cscr^{(k,\alpha)}$ if $k\ge 1$, 
and of class $\Cscr^{(1,\alpha)}$ if $k=0$.

By \cite[Theorem 4.1]{Forstneric2024Runge} there are a constant 
$c>0$ and a function $h:B_c \times \overline \Omega\to \C$ such that
for every $\mu\in B_c$, $h_\mu=h(\mu,\cdotp):\overline \Omega\to\C$ 
is a $J_\mu$-holomorphic immersion of class 
$\Cscr^{(k+1,\alpha)}(\overline \Omega)$ depending analytically on $\mu$.
We recall the proof since we shall use a similar idea 
in the sequel. The function $f=h_\mu$ must solve 
the Beltrami equation $f_{\bar z} = \mu f_z$.
We look for a solution in the form $f=z|_{\overline \Omega} +P(\phi)$ with
$\phi\in \Cscr^{(k,\alpha)}(\overline \Omega)$.
Note that $\phi=0$ corresponds to $f=z|_{\overline\Omega}$. 
We have that 
\[
	f_{\bar z}= P(\phi)_{\bar z}  = \phi
	\quad \text{and} \quad 
	f_z = 1 + P(\phi)_z = 1+ S(\phi).
\]
Inserting in the Beltrami equation $f_{\bar z} = \mu f_z$ gives $(I - \mu S) \phi =\mu$,
where $I$ denotes the identity map on $\Cscr^{(k,\alpha)}(\overline\Omega)$. 
For $\|\mu\|_{k,\alpha}$ small enough we have that $\|\mu S\|_{k,\alpha}<1$, 
so the operator $I- \mu S$ is invertible and its bounded inverse  
depends analytically on $\mu$: 
\[ 
	\Theta(\mu) := (I - \mu S)^{-1} =
	\sum_{j=0}^\infty (\mu S)^j 
	\in  \mathrm{Lin}(\Cscr^{k,\alpha}(\overline\Omega)). 
\] 
(See Mujica \cite[29.3 Theorem]{Mujica1986}.)
Hence, $\phi=\Theta(\mu)\mu$ and we get 
the following solution $h_\mu=f$ to the 
Beltrami equation $f_{\bar z} = \mu f_z$ on $\overline\Omega$:
\[ 
	h_\mu = z|_{\overline\Omega} +  P (\Theta(\mu) \mu) 
	= z|_{\overline\Omega} + P((I - \mu S)^{-1} \mu)
	 \in \Cscr^{(k+1,\alpha)}(\overline\Omega).
\] 
Note that $h_\mu$ depends analytically on $\mu$. For $\|\mu\|_{k,\alpha}$ 
small enough, $h_\mu$ is so close to $h_0=z|_{\overline\Omega}$ in
$\Cscr^{(k+1,\alpha)}(\overline\Omega)$ that it is an immersion.
Hence, for $c>0$ small enough, 
$\{\theta_\mu=dh_\mu\}_{\mu\in B_c}$
is a family of nowhere vanishing holomorphic 
$1$-forms on $(\Omega,J_b)$ of class 
$\Cscr^{(k,\alpha)}(\overline\Omega)$ with analytic dependence on $\mu$. 
(The existence of a family $\{\theta_b\}_{b\in B}$ of globally defined nowhere vanishing holomorphic $1$-forms on $(X,J_b)$ of class $\Cscr^{l,(k,\alpha)}$ was proved in \cite[Theorem 7.1]{Forstneric2024Runge} under a stronger assumption on the parameter space $B$.) 
The conjugate $\bar \theta_\mu=d\overline{h_\mu}$ is a nowhere vanishing 
antiholomorphic $(0,1)$-form with respect to $J_\mu$ for every 
$\mu\in B_c$. Thus, every family $\{\beta_\mu\}_{\mu\in B_c}$ 
on $\overline\Omega$, where $\beta_\mu$ is a $(0,1)$-form 
with respect to $J_\mu$, is of the form 
$\beta_\mu=u_{\mu}\, d\overline{h_\mu}$ for a family of functions 
$u_\mu: \overline\Omega\to \C$, $\mu\in B_c$. If the family 
$\{\beta_\mu\}_{\mu\in B_c}$ is of class 
$\Cscr^{l,(k,\alpha)}(B_c\times\overline\Omega)$
then $\{u_\mu\}_{\mu\in B_c}$ is of the same class, and vice versa.

We shall express the $\dibar_{J_\mu}$-equation \eqref{eq:dibar2} as a 
nonhomogeneous Beltrami equation with respect to the immersion $z$.
For $\mu\in B_c$ we can uniquely express any complex $1$-form 
$\beta$ on $\overline\Omega$ as 
\begin{equation}\label{eq:beta}
	\beta=A dz+Bd\bar z = A_\mu dh_\mu + B_\mu d\overline{h_\mu}.
\end{equation}
Note that $\beta$ is of class $\Cscr^{(k,\alpha)}(\overline\Omega)$
if and only if the coefficients $A,B,A_\mu,B_\mu:\overline\Omega\to\C$
are of this class. We now express $A_\mu$ and $B_\mu$ in terms of 
the functions $A,B,\mu$, and 
\begin{equation}\label{eq:gmu}	
	g_\mu := (h_\mu)_z \in \Cscr^{(k,\alpha)}(\overline\Omega).
\end{equation}
We have that 
\[
	dh_\mu = (h_\mu)_z \, dz + (h_\mu)_{\bar z} \, d\bar z
	= g_\mu\, dz + \mu g_\mu \, d\bar z
\]
where the second identity follows from the Beltrami equation
$(h_\mu)_{\bar z}=\mu (h_\mu)_z$. From this and \eqref{eq:beta} it follows that
\begin{eqnarray*}
	A dz+Bd\bar z &=& A_\mu (g_\mu dz + \mu g_\mu d\bar z)
	+ B_\mu (\overline {\mu g_\mu}\, dz + \overline{g_\mu} \, d\bar z) \\
	&=& (A_\mu g_\mu + B_\mu \overline{\mu g_\mu}) dz
	+ (A_\mu \mu g_\mu + B_\mu \overline{g_\mu}) d\bar z
\end{eqnarray*}
and hence
\[
	A = A_\mu g_\mu + B_\mu \overline{\mu g_\mu},
	\qquad
	B = A_\mu \mu g_\mu + B_\mu \overline{g_\mu}.
\]
Solving these equations on $A_\mu$ and $B_\mu$ gives
\[
	A_\mu = \frac{A  - \bar \mu B}{(1-|\mu|^2)g_\mu}, 
	\qquad 
	B_\mu = \frac{B  - \mu A}{(1-|\mu|^2) \overline{g_\mu}}.
\]
Taking $\beta=df$ we have $A=f_z$, $B=f_{\bar z}$, $A_\mu = f_{h_\mu}$, 
and $B_\mu= f_{\overline{h_\mu}}$. Inserting these quantities in 
the above expression for $B_\mu$ shows that the 
nonhomogeneous Cauchy--Riemann equation 
\[
	\dibar_{J_\mu} f = \beta_\mu = u_\mu d\overline{h_\mu} 
	\ \Longleftrightarrow \ 
	f_{\overline{h_\mu}} = 
	\frac{f_{\bar z}-\mu f_z} {(1-|\mu|^2) \overline{g_\mu}} = u_\mu
\]
is equivalent to the nonhomogeneous Beltrami equation 
\begin{equation}\label{eq:dibarmu}
	f_{\bar z}-\mu f_z = (1-|\mu|^2) \overline{g_\mu} \, u_\mu
\end{equation}
with $g_\mu$ given by \eqref{eq:gmu}. 
The regularity conditions in the proposition imply that the right hand side is of class $\Cscr^{l,(k,\alpha)}$ on $B_c\times \overline\Omega$. 
We look for a solution of \eqref{eq:dibarmu} in the form $f=f(\mu)=P(\phi)$ with 
$\phi\in \Cscr^{(k,\alpha)}(\overline\Omega)$ to be determined. 
Inserting $f_{\bar z}=P(\phi)_{\bar z}=\phi$ and $f_{z}=P(\phi)_z=S(\phi)$
into \eqref{eq:dibarmu} gives
\[
	f_{\bar z}-\mu f_z = (I-\mu S)\phi = 
	(1-|\mu|^2) \overline{g_\mu} \, u_\mu.
\]
For $\|\mu\|_{k,\alpha}$ small enough the operator $I-\mu S$ is
invertible and we obtain 
\[
	\phi=\phi_\mu = 
	(I-\mu S)^{-1}\left((1-|\mu|^2) \overline{g_\mu} \, u_\mu \right).
\] 
Since the bounded linear operator 
$(I-\mu S)^{-1}\in \mathrm{Lin}(\Cscr^{k,\alpha}(\overline \Omega))$
is analytic in $\mu$ and $(1-|\mu|^2) \overline{g_\mu} \, u_\mu 
\in \Cscr^{l,(k,\alpha)}(B_c\times\overline \Omega)$, the map
$(\mu,x)\mapsto \phi_\mu(x)$ belongs to
$\Cscr^{l,(k,\alpha)}(B_c\times\overline \Omega)$.
Finally, the solution of \eqref{eq:dibarmu} is $f_\mu = P(\phi_\mu)$,  
and the map $(\mu,x)\to f_\mu(x)$ belongs to
$\Cscr^{l,(k+1,\alpha)}(B_c\times \overline \Omega)$. 
\end{proof}

This completes the proof of Theorem \ref{th:dibar}. 
\end{proof}

\begin{proof}[Proof of Corollary \ref{cor:dibar}]
Global solvability of \eqref{eq:dibar1} is 
obtained by exhausting $X$ by an increasing family 
of relatively compact, smoothly bounded Runge domains, solving the 
equation \eqref{eq:dibar1} on each of them by using
Theorem \ref{th:dibar}, and applying the Runge 
approximation theorem on families of open Riemann
surfaces (see \cite[Theorem 1.1]{Forstneric2024Runge})  
at every step of the induction to ensure convergence of solutions. 
The straightforward details can be found in the preprint 
\cite[V3, proof of Theorem 9.1]{Forstneric2024Runge}.
One essentially follows the standard scheme in the proof of 
Cartan's Theorem B \cite[Section VIII.14]{GunningRossi1965}.
\end{proof}

%
%
\begin{remark}
\label{rem:Cartan}
(A) If $B$ is a manifold of class $\Cscr^l$ with 
$0<l\le k+1$ then the space $Z=B\times X$ in 
Theorem \ref{th:dibar}, endowed with an atlas 
of class $\Cscr^{l,(k+1,\alpha)}$ given by 
\cite[Theorem 4.1]{Forstneric2024Runge},
is a {\em mixed manifold} of class $\Cscr^l$
in the sense of 
Jurchescu \cite{Jurchescu1979,Jurchescu1988RRMPA},
and a Levi-flat CR manifold of CR-dimension one in the 
sense of Cauchy--Riemann geometry \cite{BaouendiEbenfeltRothschild1999}. 
(Mongodi and Tomassini \cite{MongodiTomassini2019} also call 
such manifolds {\em semiholomorphic foliations}. Apparently they
were not aware of Jurchescu's work.) 
In Jurchescu's papers, maps which are holomorphic on complex leaves of a mixed manifold are called {\em morphic}, while in CR geometry they are called CR maps. The Runge approximation theorem 
\cite[Theorem 1.1]{Forstneric2024Runge} 
shows that such $Z$ is also a {\em Cartan manifold} of class $\Cscr^{l}$ 
in the sense of \cite[Sect.\ 6]{Jurchescu1988RRMPA}.
Cartan manifolds are analogues of Stein manifolds in 
the category of mixed manifolds. 
%
%
Solvability of the tangential $\dibar$-complex on $\Cscr^\infty$  
smooth Cartan manifolds was shown by Jurchescu in
\cite[Sect.\ 3]{Jurchescu1994} by using sheaf-theoretic approach,
similar to the one in the classical theory of Stein manifolds. 

\noindent 
(B) There are results in the literature concerning the $\dibar$-equation 
on moving families of domains, also in higher dimensional complex manifolds; 
see Diederich and Ohsawa \cite{DiederichOhsawa1991}, 
Cho and Choi \cite{ChoChoi2008}, 
Gong and Kim \cite[Theorem 4.5]{GongKim2018},
Simon \cite{Simon2019JGA}, 
Kruse \cite{Kruse2020}, 
among others. By pulling back a complex structure by a
family of diffeomorphisms, the case of moving domains
is changed to the variation of the complex structure on a fixed domain. 
Conversely, Hamilton's theorem \cite{Hamilton1977} 
(see also \cite[Theorem 1.13]{GreeneKrantz1982} 
for a simpler proof and \cite[Corollary 4.5]{Forstneric2024Runge} for 
domains in Riemann surfaces) shows that a small smooth
integrable variation of the complex structure on a 
smoothly bounded strongly pseudoconvex domain can be realised
by a smooth variation of the domain in the ambient Stein manifold.
In \cite{GreeneKrantz1982}, Greene and Krantz 
used methods introduced by Kohn \cite{Kohn1963,Kohn1964}
to study stability of the $\dibar$-Neumann operator and 
the canonical (Kohn) solution 
to the $\dibar$-equation under small integrable perturbations 
of a complex structure $J$ on a compact strongly 
$J$-pseudoconvex domain $\overline M$.
Assuming that the boundary $bM$ is of class $2s+5$ for some $s\ge 1$,
they obtained continuous dependence of the Neumann operator
$N_J$ in the Sobolev $L^2$-space $W^{s}$ 
under small variations of $J$ of class $\Cscr^{2s+5}$, and hence
continuous dependence of solutions of the $\dibar_J$-equation 
in $W^{s-1}$ \cite[Theorems 3.9, and 3.10]{GreeneKrantz1982}.
It is not clear whether these results are (close to) optimal, 
and we could not find results in the literature with better than continuous 
dependence of solutions on the complex structure.
\end{remark}

%
%

Theorem \ref{th:dibar} implies the vanishing of the 
Dolbeault cohomology on families of open Riemann surfaces. 
To explain this, assume that 
$B$, $X$, and $\{J_b\}_{b\in B}$ are as in Theorem \ref{th:dibar}, 
where the family $J_b$ is of class $\Cscr^{l,(k,\alpha)}$ for some 
$0\le l\le k+1$ and $0<\alpha<1$.  
Denote by $\Oscr$ the sheaf of germs of functions $f$ of class 
$\Cscr^{l}$ on $Z=B\times X$ which are fibrewise holomorphic,
that is, $f_b=f(b,\cdotp)$ is $J_b$-holomorphic for each $b\in B$. 
By \cite[Lemma 5.6]{Forstneric2024Runge},
$\Oscr$ is a subsheaf of the sheaf $\Cscr^{l,(k+1,\alpha)}$.
These are sheaves of unital abelian rings, in particular, of abelian groups. 

%
%
\begin{proposition}\label{prop:Dolbeault}
(Assumptions as above.) 
$H^q(Z,\Oscr)=0$ for all $q=1,2,\ldots$.
\end{proposition}

\begin{proof}
Consider the sequence of homomorphisms of sheaves of abelian groups
\begin{equation}\label{eq:resolution}
	0\lra \Oscr \longhookrightarrow \Cscr^{l,(k+1,\alpha)}
	\stackrel{\dibar}{\lra} \Cscr^{l,(k,\alpha)}_{(0,1)} \lra 0,
\end{equation}
where $\Cscr^{l,(k,\alpha)}_{(0,1)}$ is the sheaf of germs 
of $(0,1)$-forms of class $\Cscr^{l,(k,\alpha)}$ on the fibres 
$Z_b=(X,J_b)$ and $\dibar$ is the operator 
which equals $\dibar_{J_b}$ on $Z_b$ for every $b\in B$. 
By Theorem \ref{th:dibar} the sequence \eqref{eq:resolution} is exact.
The second and the third sheaf in \eqref{eq:resolution} 
are fine sheaves (they admit partitions of unity), so they are acyclic,
i.e., their cohomology groups of order $\ge 1$ vanish. It follows that
\[ 
	H^1(Z,\Oscr)= \Gamma\big(Z,\Cscr^{l,(k,\alpha)}_{(0,1)}\big)/
	\dibar \, \Gamma\big(Z,\Cscr^{l,(k+1,\alpha)}\big)
\] 
and $H^q(Z,\Oscr)=0$ for $q\ge 2$ 
(see \cite[Chapter VI]{GunningRossi1965}). 
Here, $\Gamma$ denotes the space of global sections of a sheaf. 
The quotient group on the right hand side above
vanishes by Theorem \ref{th:dibar}.
\end{proof}

%
%
%
\section{The Oka principle for line bundles
on families of open Riemann surfaces}

%
%
Every holomorphic vector bundle on an 
open Riemann surface is holomorphically trivial by the Oka--Grauert principle; see \cite{Oka1939,Grauert1958MA} 
and \cite[Theorem 5.3.1]{Forstneric2017E}.
We now show that Proposition \ref{prop:Dolbeault} implies
the Oka principle for complex line bundles on families
of open Riemann surfaces; see Theorem \ref{th:OPLB}.
We also discuss vector bundles of higher rank.

Let $B$ be a paracompact Hausdorff space, $X$ be a smooth
open surface, and $\{J_b\}_{b\in B}$ be a continuous family of complex
structures on $X$ of class $\Cscr^{\alpha}$ for some $0<\alpha<1$.
We let $Z=B\times X$ with the fibre $Z_b=\{b\}\times X$ endowed
with the complex structure $J_b$ for each $b\in B$.
Denote by $\Cscr=\Cscr_Z$ and $\Oscr=\Oscr_Z$
the sheaves of germs of continuous resp.\ 
fibrewise holomorphic functions on $Z=B\times X$.
These are sheaves of unital abelian rings, with $\Oscr_Z$ a subsheaf
of $\Cscr_Z$. Furthermore, $\Oscr^*\subset \Oscr$ and 
$\Cscr^*\subset \Cscr$ denote the subsheaves consisting 
of germs with nonzero values; 
these are sheaves of (multiplicative) abelian groups.
Likewise, for a complex Lie group $G$, $\Oscr_Z^G\subset \Cscr_Z^G$ 
denote the sheaves of germs of maps $Z\to G$ of respective classes; 
these are sheaves of groups. 

%
%
\begin{definition}\label{def:FHVB}
Let $X$, $\{J_b\}_{b\in B}$, and $Z=B\times X$ be 
as above, and denote by $Z_b=(X,J_b)$ the fibre of $Z$ over $b\in B$.
A topological complex vector bundle $E \to Z$
is said to be fibrewise holomorphic 
if for any $b\in B$ the restriction $E|Z_b\to Z_b$ is a holomorphic vector
bundle on the Riemann surface $(X,J_b)$.
\end{definition}

Let $Gl(r,\C)$ denote the group of invertible $r\times r$ matrices.
We have the following observation.

\begin{proposition}\label{prop:FHVB}
A complex vector bundle $E\to Z=B\times X$ of rank $r$
is fibrewise holomorphic if and only if it admits a transition cocycle 
consisting of sections of the sheaf $\Oscr_Z^{Gl(r,\C)}$.
\end{proposition}

In the case of line bundles we have $r=1$ and $Gl(r,\C)=\C^*$, 
so $\Oscr^{Gl(r,\C)}=\Oscr^*$.

\begin{proof}
A vector bundle $E\to Z$ which admits a transition cocycle 
consisting of sections of the sheaf $\Oscr_Z^{Gl(r,\C)}$ is 
fibrewise holomorphic by the definition of the sheaf $\Oscr_Z$.

To prove the converse, it suffices to show that every point $z_0=(b_0,x_0)\in Z$
has a neighbourhood $U\subset Z$ such that the restricted
bundle $E|U$ admits $r=\mathrm{rank}\, E$ linearly independent
continuous fibrewise holomorphic sections. Such $r$-tuples of 
sections trivialise the bundle over $U$, and the transition maps 
between such local frames are fibrewise holomorphic maps into $Gl_r(\C)$.

To find such local frames, fix a point $(b_0,x_0)\in Z$ and 
smoothly bounded relatively compact domains 
$\Omega\Subset \Omega'\Subset X$ such that $\Omega$ is a
contractible neighbourhood of $x_0$.
By \cite[Corollary 4.5]{Forstneric2024Runge} there are a neighbourhood 
$B_0\subset B$ of $b_0$ and a continuous map 
$\Phi:B_0\times \Omega'\to B_0\times X$ of the form 
\begin{equation}\label{eq:Hamilton}
	\Phi(b,x)=(b,\Phi_b(x)),\quad b\in B_0,\ x\in \Omega' 
\end{equation}
such that $\Phi_{b_0}$ is the identity on $\Omega'$ and 
$\Phi_b:\Omega'\to\Phi_b(\Omega')\subset X$ 
is a $(J_b,J_{b_0})$-biholomorphic 
map in $\Cscr^{(1,\alpha)}(\Omega',X)$ which is continuous 
with respect to $b\in B_0$. Shrinking $B_0$ around $b_0$, 
we may assume that $\overline \Omega \subset \Phi_b(\Omega')$
for all $b\in \overline{B_0}$. The inverse map $\Psi=\Phi^{-1}$ is a
homeomorphism defined on $\overline{B_0 \times  \Omega}$,
and $\Psi_b=\Psi(b,\cdotp)=\Phi_b^{-1}$ is $(J_{b_0},J_b)$-holomorphic 
for every $b\in \overline{B_0}$. Hence, the vector bundle 
$E'=\Psi^*(E)| \,\overline{B_0 \times  \Omega}\to \overline{B_0 \times  \Omega}$ is fibrewise
holomorphic with respect to the complex structure $J_{b_0}$.
Since $\overline \Omega$ is contractible, the restricted bundle
$E'_b=E'| \{b\}\times \overline\Omega$ is holomorphically trivial 
with respect to $J_{b_0}$ for every $b\in \overline B_0$. 
Furthermore, by the stability theorem
of Leiterer \cite[Theorem 2.7, p.\ 68]{Leiterer1990}, shrinking 
$B_0$ around $b_0$ there is family 
of $J_{b_0}$-holomorphic trivialisations depending continuously on 
$b\in B_0$. This gives a local frame for the bundle $E'$ on
a neighbourhood of $(b_0,x_0)$
consisting of continuous fibrewise $J_{b_0}$-holomorphic sections. 
Their pullbacks by the map $\Phi$ \eqref{eq:Hamilton} is a 
local frame for $E$ around $(b_0,x_0)$ consisting of fibrewise 
holomorphic sections with respect to the variable complex
structures $J_b$, $b\in B_0$. This completes the proof.
\end{proof}

Denote by $\Pic(Z)\cong H^1(Z,\Oscr^*)$ the set of isomorphism 
classes of fibrewise holomorphic line bundles on $Z=B\times X$. 
We have the following Oka principle.

%
%
\begin{theorem}\label{th:OPLB}
(Assumptions as above.) 
Every topological complex line bundle on $Z=B\times X$ 
is isomorphic to a fibrewise holomorphic line bundle, 
and any two fibrewise holomorphic line bundles on $Z$
which are topologically isomorphic are also
isomorphic as fibrewise holomorphic line bundles.
Furthermore, $\Pic(Z)\cong H^2(Z,\Z)$.
\end{theorem} 

It follows in particular that if $B$ is contractible then $\Pic(Z)=0$.

\begin{proof}
The proof follows the standard argument for complex line
bundles on a Stein manifold, due to Oka 
\cite{Oka1939}; see \cite[Theorem 5.2.2]{Forstneric2017E}. 
Let $\sigma(f)= \E^{2\pi \I f}$. Consider the following
commutative diagram whose rows are 
exponential sheaf sequences and whose vertical arrows are 
the natural inclusions:
\begin{equation} \label{eq:expo-sheaf}
\begin{array}{*{9}c}
	0 & \lra & \Z  & \longhookrightarrow & \Oscr 
	& \stackrel{\sigma}{\lra} & \Oscr^*  & \lra & 1 \\
 	&    & \big\downarrow &  & \big\downarrow &                 
	& \!\!\big\downarrow & &  \\
	0 & \lra & \Z & \longhookrightarrow & \Cscr  & \stackrel{\sigma}{\lra} 
	& \Cscr^*  & \lra & 1 \\
\end{array}
\end{equation}
Since $\Cscr$ is a fine sheaf, we have  $H^q(Z,\Cscr)=0$ for all $q\in\N$.
By Proposition \ref{prop:Dolbeault} we also have $H^q(Z,\Oscr)=0$
for all $q\in\N$. Hence, the relevant part of 
the long exact sequence of cohomology groups associated 
to the diagram \eqref{eq:expo-sheaf} gives
\[ 
\begin{array}{*{8}c}
 	0  & \lra & H^1(Z,\Oscr^*) & \lra & H^2(Z;\Z) & \lra & 0 \\
        &  & \big\downarrow & & \big\Vert &  &    \\
        0   & \lra & H^1(Z,\Cscr^*) & \lra & H^2(Z;\Z) & \ \, \lra & 0 \\
\end{array}
\]
Thus, all arrows in the central square are isomorphisms. 
Since $\Pic(Z)\cong H^1(Z,\Oscr^*)$ and 
$H^1(Z,\Cscr^*)$ is the set of isomorphisms classes
of topological line bundles on $Z$, the theorem follows.
\end{proof}

%
%
The Oka principle for maps from families of open Riemann surfaces
to Oka manifolds (see \cite[Theorem 1.6]{Forstneric2024Runge}) 
allows us to extend the first part of Theorem \ref{th:OPLB} 
to vector bundles of arbitrary rank by using the approach from 
the classical Oka--Grauert theory. However, the assumptions
on the parameter space $B$ must be more restrictive for 
the cited result to imply. We state the following special case 
when $B$ is a CW complex and refer to the discussion preceding 
\cite[Theorem 1.6]{Forstneric2024Runge} for more information.


%
%
\begin{theorem}\label{th:OPvectorbundles} 
Assume that $B$ is a finite CW complex or a countable locally 
compact CW-complex of finite dimension, $X$ 
is a smooth open surface, and $\{J_b\}_{b\in B}$ is  
a continuous family of complex structures on 
$X$ of local H\"older class $\Cscr^\alpha$ for some $0<\alpha<1$.
Then, every topological vector bundle on $B\times X$ is isomorphic 
to a fibrewise holomorphic vector bundle. 
\end{theorem}

\begin{proof}
A topological vector bundle $E\to B \times X$ of rank $r$
is the pullback $f^*\U$ 
by a continuous map $f$ from $B \times X$ to a 
Grassmannian $G=Gr(r,N)$ (consisting of complex
$r$-planes in $\C^N$) of the universal rank $r$ vector bundle $\U\to G$.
(We take $N$ big enough such that $E$ embeds as a topological
vector subbundle of the trivial bundle $(B \times X)\times \C^N$;
this is possible since $B \times X$ is paracompact.)
Since $G$ is a complex homogeneous manifold, and hence an 
Oka manifold by Grauert \cite{Grauert1957II}, 
the Oka principle in \cite[Theorem 1.6]{Forstneric2024Runge}
shows that $f$ is homotopic to a fibrewise holomorphic map 
$F:B\times X\to G$ (i.e., such that $F(b,\cdotp):X\to G$ is $J_b$-holomorphic 
for every $b \in B$). The pullback $F^*\U\to B\times X$ is then a fibrewise 
holomorphic vector bundle topologically isomorphic to $f^*\U$. 
\end{proof}

%
%
\begin{remark}\label{rem:OPvectorbundles}
Under the assumptions in Theorem \ref{th:OPvectorbundles},
it is likely the case that any two fibrewise holomorphic vector bundles on 
$B\times X$ that are topologically isomorphic are also isomorphic 
as fibrewise holomorphic vector bundles.
The idea is to follow the classical case of this result for a single
complex structure on a Stein manifold $X$, due to Grauert
\cite{Grauert1958MA}; see also the expositions 
by Cartan \cite{Cartan1958}, Leiterer \cite{Leiterer1990}, and 
\cite[Theorem 5.3.1]{Forstneric2017E}. 
To complete the proof, we would need an Oka principle 
for sections of fibrewise holomorphic principal bundles
on $B\times X$, thereby extending the Oka principle
for maps from $B\times X$ to Oka manifolds
in \cite[Theorem 1.6]{Forstneric2024Runge}.
Unfortunately I was unable to handle certain technical 
issues concerning gluing of families of fibrewise holomorphic 
sections in this situation, so this problem must be postponed.

Note that Mongodi and Tomassini \cite{MongodiTomassini2019}  
obtained the Oka principle for more general CR vector bundles 
on certain real analytic Levi-flat submanifolds of complex 
Euclidean spaces by reducing the problem to the 
Oka--Grauert theorem \cite{Grauert1958MA}.
Our techniques use considerably less regularity.
\end{remark}

%
%
%
%
\medskip 
\noindent {\bf Acknowledgements.} 
Research was supported by the European Union 
(ERC Advanced grant HPDR, 101053085) and grants P1-0291 and 
N1-0237 from ARIS, Republic of Slovenia. 
I wish to thank Finnur L\'arusson for helpful conversations concerning 
Remark \ref{rem:OPvectorbundles}.



\end{document}